\documentclass[a4paper, 11pt]{amsart}
\usepackage{longtable, stackrel, multicol}
\usepackage{bm,amsfonts,amsmath,amssymb,mathrsfs,bbm,amsthm}
\usepackage{graphicx, color,hyperref,eurosym, eucal,palatino}
\usepackage{scrextend}
\definecolor{darkblue}{rgb}{0.2,0.2,0.6}
\definecolor{superdarkblue}{rgb}{0.2,0.2,0.3}
\hypersetup{colorlinks=true, linkcolor=darkblue,citecolor=darkblue,
urlcolor = superdarkblue}
\definecolor{citegreen}{rgb}{0.2,0.2,0.6}

\usepackage[normalem]{ulem}

\hypersetup{colorlinks=true, linkcolor=blue,citecolor=citegreen}
\definecolor{citegreen}{rgb}{0.2,0.2,0.6}
\definecolor{darkred}{rgb}{0.6,0.2,0.2}

\makeatletter
\def\section{\@startsection{section}{1}\z@{.9\linespacing\@plus\linespacing}%
	{.7\linespacing} {\fontsize{13}{14}\selectfont\bfseries\centering}}
\def\paragraph{\@startsection{paragraph}{4}%
	\z@{0.3em}{-.5em}%
	{$\bullet$ \ \normalfont\itshape}}

\@addtoreset{equation}{section}
\makeatother

\def\bm1{\mathbbm{1}}

\newcommand\dl{\delta}

\newcommand\Qb{\sfQ_\beta}
\newcommand\Db{\sfD_\beta}
\newcommand\wDb{\sfD_{\beta,V}}
\newcommand\wB{\sfB_{\aa,V}}
\newcommand\Gb{\sfG_{\aa,\beta}}
\newcommand\Pb{\sfP_{\aa,\beta}}

\newcommand{\ie}{\emph{i.e.}}
\newcommand{\eg}{\emph{e.g.}}
\newcommand{\cf}{\emph{cf.}}

\newcommand{\R}{\mathbb{R}}
\newcommand{\e}{\mathrm{e}}
\renewcommand{\aa}{\alpha}
\newcommand{\lm}{\lambda}

\newcommand{\Op}{\sfH_{\alpha ,\beta}}

\newcommand{\beq}{\begin{equation} \begin{split}}
\newcommand{\eeq}{\end{split} \end{equation}}

\newcommand{\comm}[1]{}

\def\sfG{\mathsf{G}}

\def\sfH{\mathsf{H}}

\def\bm1{\mathbbm{1}}

\def\s{\sigma}

\def\p{\partial}

\newcommand\dd{\mathrm{d}}

\renewcommand{\iff}{\textit{if, and only if,}\,}

\def\arr{\rightarrow}

\newcommand{\sfJ}{\mathsf{J}}

\def\tt{\theta}
\def\lm{\lambda}

\def\s{\sigma}
\def\sd{\sigma_{\rm d}}
\def\sess{\sigma_{\rm ess}}

\def\ii{{\mathsf{i}}}
\def\p{\partial}

\def\kp{\kappa}

\def\sfH{\mathsf{H}}

\def\sfP{\mathsf{P}}

\def\sfP{\mathsf{P}}

\newcounter{counter_a}
\newenvironment{myenum}{\begin{list}{{\rm(\roman{counter_a})}}%
{\usecounter{counter_a}
\setlength{\itemsep}{1.ex}\setlength{\topsep}{0.8ex}
\setlength{\leftmargin}{5ex}\setlength{\labelwidth}{5ex}}}{\end{list}}

\usepackage[latin1]{inputenc}
\usepackage[T1]{fontenc}

\numberwithin{figure}{section}
\numberwithin{equation}{section}
\theoremstyle{plain}
\newtheorem*{thm*}{Theorem}
\newtheorem{thm}{Theorem}[section]

\newtheorem{lem}[thm]{Lemma}
\newtheorem{prop}[thm]{Proposition}

\newtheorem{cor}[thm]{Corollary}

\theoremstyle{remark}

\theoremstyle{plain}

\newcommand{\beu}{\begin{equation*}}
\newcommand{\eeu}{\end{equation*}}
\newcommand{\besu}{\begin{equation*}
\begin{aligned}}
\newcommand{\eesu}{\end{aligned}
\end{equation*}}
\newcommand{\bes}{\begin{equation}
\begin{aligned}}
\newcommand{\ees}{\end{aligned}
\end{equation}}

\newcommand\cB{\mathcal B}
\newcommand\cD{\mathcal D}
\newcommand\cF{\mathcal F}

\newcommand\cH{\mathcal H}

\newcommand\cS{\mathcal S}

\newcommand\frh{\mathfrak h}

\newcommand\ov{\overline}
\newcommand\wt{\widetilde}
\newcommand\wh{\widehat}

\newcommand\void[1]{}

\def\ov{\overline}

      \def\dC{{\mathbb C}}

   \def\dN{{\mathbb N}}   
      \def\dR{{\mathbb R}}

\def\sfA{{\mathsf A}}   \def\sfB{{\mathsf B}}   
\def\sfD{{\mathsf D}}      
\def\sfG{{\mathsf G}}   \def\sfH{{\mathsf H}}   \def\sfI{{\mathsf I}}
\def\sfJ{{\mathsf J}}      \def\sfL{{\mathsf L}}
\def\sfM{{\mathsf M}}   \def\sfN{{\mathsf N}}   
\def\sfP{{\mathsf P}}   \def\sfQ{{\mathsf Q}}   \def\sfR{{\mathsf R}}

   \def\cB{{\mathcal B}}   
\def\cD{{\mathcal D}}      \def\cF{{\mathcal F}}
   \def\cH{{\mathcal H}}   
      
      \def\cO{{\mathcal O}}
      
\def\cS{{\mathcal S}}   \def\cT{{\mathcal T}}   \def\cU{{\mathcal U}}
\def\cV{{\mathcal V}}

\newcommand{\Tr}{\mathrm{Tr}\,}

\newcommand{\supp}{\mathop{\mathrm{supp}}\nolimits}

\newtheorem{claim}{Claim}[section]

\newtheorem{definition}[claim]{Definition}

\definecolor{DarkGreen}{rgb}{0,0.5,0.1}

\newcommand\soutD{\bgroup\markoverwith
{\textcolor{DarkGreen}{\rule[.5ex]{2pt}{1pt}}}\ULon}
\newcommand{\Hm}[1]{\leavevmode{\marginpar{\tiny%
$\hbox to 0mm{\hspace*{-0.5mm}$\leftarrow$\hss}%
\vcenter{\vrule depth 0.1mm height 0.1mm width \the\marginparwidth}%
\hbox to 0mm{\hss$\rightarrow$\hspace*{-0.5mm}}$\\\relax\raggedright
#1}}}

\begin{document}

\title[Asymptotics of the bound state
    induced by \boldmath{$\delta$}-interaction on a weakly deformed plane]{Asymptotics
        of the bound state
    induced by \boldmath{$\delta$}-interaction supported on a weakly deformed plane}

\author{Pavel Exner}

%

\address{Department of Theoretical Physics, Nuclear Physics Institute, Czech Academy of Sciences, 25068 \v Re\v z near Prague, Czechia, and Doppler Institute for Mathematical Physics and Applied Mathematics, Czech Technical University, B\v rehov\'a 7, 11519 Prague, Czechia}
\email{exner@ujf.cas.cz}

\author{Sylwia Kondej}

\address{Institute of Physics, University of Zielona G\'ora, ul.\ Szafrana 4a, 65246 Zielona G\'ora, Poland}
\email{s.kondej@if.uz.zgora.pl}

\author{Vladimir Lotoreichik}
\address{Department of Theoretical Physics, Nuclear Physics Institute, Czech Academy of Sciences, 25068 \v Re\v z near Prague, Czechia}
\email{lotoreichik@ujf.cas.cz}


\maketitle
\begin{abstract}
    In this paper we consider
    the three-dimensional Schr\"{o}dinger operator
    with a $\delta$-interaction of strength
    $\aa > 0$ supported on an unbounded
    surface parametrized by the mapping
    $\dR^2\ni x\mapsto (x,\beta f(x))$, where $\beta \in [0,\infty)$ and $f\colon \dR^2\arr\dR$, $f\not\equiv 0$, is a $C^2$-smooth,
    compactly supported function. The surface supporting
    the interaction can be viewed as a local deformation of the plane.
    It is known that the essential spectrum of
    this Schr\"odinger operator
    coincides with $[-\frac14\aa^2,+\infty)$. We prove
    that for all sufficiently small $\beta > 0$ its discrete
    spectrum is non-empty and consists
    of a unique simple eigenvalue. Moreover, we obtain
    an asymptotic
    expansion of this eigenvalue in the limit $\beta \arr 0+$.
    In particular,
    this eigenvalue  tends to $-\frac14\aa^2$ exponentially fast as $\beta\arr 0+$.
\end{abstract}
%
\section{Introduction}\label{Sec.intro}
\subsection{Motivation}
Various physical systems can be effectively described by
Schr\"{o}dinger operators with $\delta$-interactions
supported on sets of zero Lebesgue measure.
To mention just a few, these operators are used:
\begin{myenum}
    \item [--] in mesoscopic physics in the model of leaky quantum graphs~\cite[Chap. 10]{EK};
    \item [--] for the description of atoms in strong magnetic fields~\cite{BD06};
    \item [--] in the theory of semiconductors as a model for excitons~\cite{HKPC17};
    \item [--] for the analysis of high contrast photonic crystals~\cite{FK96, HL17}.
\end{myenum}
One can expect that this list will keep expanding, in particular,
with the simplicity and versatility of the model in mind. This is
certainly a motivation to investigate its properties by rigorous
mathematical means.

One of the most traditional problems concerns the relation between
the geometry of the support of the $\delta$-interaction and the
spectrum of the corresponding Schr\"odinger operator; see the
review~\cite{E08}, the monograph~\cite{EK}, and the references
therein. A prominent particular question, addressed in numerous
papers (see \eg~\cite{BEL14a, EI01, EK03, EL17, OP16, P15}), is to
analyze whether bound states below the threshold of the essential
spectrum are induced by an attractive $\delta$-interaction supported
on an unbounded, asymptotically flat hypersurface.

In the two-dimensional setting, this question is answered
affirmatively in~\cite{EI01}, provided that the asymptotically
straight curve is not a straight line.  In the space dimension $d
\ge 4$, a circular conical surface is a non-trivial
example~\cite{LO16} of an asymptotically flat hypersurface such that
an attractive $\delta$-interaction of any strength, supported on it,
induces no bound states. Apparently, the three-dimensional case
happens to be the most subtle. In this space dimension, existence of
bound states (in fact, infinitely many of them) is shown
in~\cite{BEL14a, LO16, OP16} for all interaction strengths in the
geometric setting of conical surfaces, which is a special class of
asymptotically flat surfaces. On the other hand, for the most
natural geometric setting of locally deformed planes, existence of
at least one bound state below the threshold is proven
in~\cite{EK03} only in the strong-coupling regime. For the same
geometry, the question of existence of bound states below the
threshold for an arbitrary strength of an attractive
$\delta$-interaction still remains open and challenging.

The aim of this paper is to make one more step towards the complete
answer to this open question. Specifically, we prove the existence
of bound states induced by $\delta$-interactions supported on
locally deformed planes, in the small deformation limit. As a
by-product of the proof we obtain that for a sufficiently small
deformation the discrete spectrum consists of unique simple
eigenvalue. Moreover, we derive an asymptotic expansion of this
eigenvalue in terms of the profile of the deformation.
\subsection*{Notations}
Throughout the paper $g(\beta,\delta) = o_{\rm u}(h(\beta))$ and
$g(\beta,\delta) = \cO_{\rm u}(h(\beta))$ denote the standard asymptotic
notations in the limit $\beta \arr 0+$, which are additionally
uniform in $\delta \in [0,1]$. For a Hilbert space $\cH$ we denote
by $\cB(\cH)$ the space of bounded, everywhere defined linear
operators in $\cH$. We denote by
$(L^2(\R^d),(\cdot,\cdot)_{L^2(\R^d)})$ (respectively, by
$(L^2(\R^d;\dC^d),(\cdot,\cdot)_{L^2(\R^d;\dC^d)})$) the usual
$L^2$-spaces over $\R^d$, $d \in\dN$, of scalar-valued
(respectively, vector-valued) functions. By $\cF\colon
L^2(\dR^2)\arr L^2(\dR^2)$ we abbreviate the unitary
Fourier-Plancherel operator; with a slight abuse of terminology we
will refer to it as to Fourier transformation in $\dR^2$. In the
same vein, for any $\psi\in L^2(\dR^2)$ its Fourier transform
$\cF\psi$ will be denoted by $\wh \psi\in L^2(\dR^2)$. By
$H^1(\R^d)$ we denote the first order $L^2$-based Sobolev space over
$\R^d$, $d \in \dN$. For a $C^2$-smooth surface $\Gamma\subset\dR^3$,
$(L^2(\Gamma),(\cdot,\cdot)_{L^2(\Gamma)})$ is the usual $L^2$-space
over $\Gamma$, where the inner product $(\cdot,\cdot)_{L^2(\Gamma)}$
is introduced via the canonical Hausdorff measure $\s(\cdot)$ on
$\Gamma$; \cf~\cite[App.~C.8]{Le}. For an open interval $I\subset\dR$,
the operator-valued function $I\ni \dl\mapsto \sfB(\dl)\in\cB(\cH)$ is real analytic
if for any $\phi,\psi\in\cH$ the scalar-valued function
$I\ni \dl\mapsto (\sfB(\dl)\phi,\psi)_\cH$ is real analytic in the usual sense.

\subsection{The spectral problem for \boldmath{$\delta$}-interaction supported on a locally deformed plane}\label{sec:pre}
Let $\Gamma = \Gamma_\beta(f)\subset\R^3$, with $\beta \in
[0,\infty)$, be an unbounded surface given by
\begin{equation}\label{eq:Gamma}
    \Gamma
    :=
    \big\{
    (x_1,x_2,x_3)\in\R^3\colon x_3 = \beta f (x_1, x_2)\big\}\subset \R^3\,,
\end{equation}
where $f\colon \dR^2\arr \dR$ ($f\not\equiv 0)$ is a $C^2$-smooth,
compactly supported function. The surface $\Gamma$ can be viewed as
a local deformation of the plane $\dR^2\times\{0\}$. 
We also point out that in view of the identity $\Gamma_{-\beta}(f) =
\Gamma_\beta(-f)$ it is enough to consider non-negative values of
$\beta$ only. In what follows we set $\cS := \supp f$ and denote by
$L_f > 0$ the Lipschitz constant of $f$; \ie~the minimal positive
number such that $|f(x) - f(y)|\le L_f|x-y|$ holds for all
$x,y\in\dR^2$. By the mean-value theorem we infer that the
inequality $|\nabla f| \le L_f$ holds pointwise. Taking the
smoothness of $\Gamma$ into account, it is not difficult to check
that the mapping $\Omega \mapsto \s(\Omega \cap \Gamma)$ defines a
measure on $\R^3$, which belongs to the \emph{generalized Kato
class}; \cf~\cite[Sec. 2]{BEKS}.

Let a constant $\aa > 0$ be fixed.
According to~\cite[Sec. 2]{BEKS} and also to~\cite[Prop. 3.1]{BEL14b}, the symmetric quadratic form
\begin{equation}\label{eq:form}
    H^1(\R^3) \ni u\mapsto \frh_{\aa,\beta}[u]
    :=
    \|\nabla u\|^2_{L^2(\R^3;\dC^3)}
            -\aa\| u|_\Gamma\|^2_{L^2(\Gamma)}\,,
\end{equation}
is closed, densely defined, symmetric, and semi-bounded in $L^2(\R^3)$;
here $u|_{\Gamma}$ denotes the trace of $u$ onto $\Gamma$.
Recall that the trace map $H^1(\dR^3)\ni u\arr  u|_{\Gamma} \in L^2(\Gamma)$ is well defined and continuous  \cite[Thm. 3.38]{McL}. Now we are in position to define the Hamiltonian
with $\delta$-interaction supported on $\Gamma$, the main object of the present paper.
\begin{definition}\label{def:Op}
    The self-adjoint Schr\"odinger operator $\sfH_{\aa,\beta}$
    in $L^2(\R^3)$ corresponding to the formal differential expression $-\Delta-\aa\, \delta(x-\Gamma)$, $\aa  > 0$, is defined
    via the first representation theorem~\cite[Thm. VI 2.1]{Kato} as associated with the quadratic form $\frh_{\aa,\beta}$ in~\eqref{eq:form}.
\end{definition}
The surface $\Gamma$ is referred to as the support of the
$\delta$-interaction and the constant $\aa > 0$ is usually called
the strength of this interaction. Schr\"odinger operators with
$\delta$-interactions supported on locally deformed planes were
first investigated in~\cite{EK03} and then subsequently
in~\cite{BEHL17, BEL14b, E17}. In the following proposition we collect
some previously known fundamental spectral properties of $\Op$.
\begin{prop}\label{thm:known}
    The spectrum of the self-adjoint operator $\Op$
    introduced in Definition~\ref{def:Op} is characterised as follows.
    \begin{myenum}
        \item $\sess(\Op) = \left [-\frac14\aa^2,+\infty\right )$.
        \item $\sd(\Op) \ne \varnothing$ for all $\beta > 0$ and
         all    $\aa > 0$ large enough.
        \item $\sd(\sfH_{\aa,0}) = \varnothing$ for $\beta = 0$
        and all $\aa > 0$.
    \end{myenum}
\end{prop}
For a proof of item~(i) see~\cite[Thm. 4.1, Rem. 4.2]{EK03} and
\cite[Thm. 4.10]{BEL14b}. A~proof of item~(ii) can be found
in~\cite[Thm. 4.3]{EK03}. The claim of~(iii) easily follows via
separation of variables. Our considerations are inspired by the open
question, whether $\sd(\Op) \ne \varnothing$ holds for all
$\aa,\beta > 0$; \cf~\cite[Problem 7.5]{E08}.

\subsection{Main result}

Informally speaking, the main result of this paper says that the
discrete spectrum of $\Op$ consists of exactly one simple eigenvalue
for all sufficiently small $\beta > 0$. Moreover, an asymptotic
expansion of this eigenvalue in terms of $\aa$, $\beta$ and of the
function $f$ is found. In order to formulate this result precisely,
we denote by $\lm_1^\aa(\beta)$ the lowest spectral point of $\Op$.
\begin{thm}\label{thm:main}
    Let $\aa > 0$ be fixed and let the self-adjoint operator $\Op$ be as in Definition~\ref{def:Op}. Set
    \begin{equation*}
        \cD_{\aa,f} := \int_{\dR^2}
        |p|^2\left (\aa^2 - \frac{2\aa^3}{\sqrt{4|p|^2 + \aa^2} + \aa}\right )|\wh f(p)|^2 \dd p > 0,
    \end{equation*}
    where $\wh f$ is the Fourier transform of $f$.
    Then $\#\sd(\Op) = 1$ holds for all sufficiently small $\beta > 0$
    and, moreover, the simple eigenvalue $\lm_1^\aa(\beta)$ admits the
    asymptotic expansion
    \begin{equation}\label{eq:main_expansion}
        \lm_1^\aa(\beta) =
        -\frac{\aa^2}{4} -
        \exp\left (-\frac{16\pi}{\cD_{\aa,f}\beta^2}\right ) \big (1+o(1)\big),\qquad \beta\arr 0+.
    \end{equation}
\end{thm}
The proof of this result relies on the Birman-Schwinger
principle~\cite{BEKS} for $\Op$. Inspired by the technique developed
in~\cite{BFKLR17, CK11, EK02, EK08, EK15}, we take the advantage of
rewriting the Birman-Schwinger condition in the perturbative form,
in which the resolvent of the two-dimensional free Laplacian
appears. A technically demanding step is to expand this new
condition with respect to the small parameter~$\beta$. Following the
strategy similar in spirit to the one used in~\cite{S76}, we derive
from this condition an implicit scalar equation on the principal
eigenvalue of $\Op$. Careful inspection of this equation yields the
existence and uniqueness of its solution for all sufficiently small
$\beta > 0$, as well as the expansion of this unique solution in the
asymptotic regime $\beta\arr 0+$. Surprisingly, an integral
representation of the relativistic Schr\"odinger
operator~\cite{IT93} arises in this asymptotic analysis. The
obtained implicit equation seems to be of an independent interest,
because it allows to extract more terms in the asymptotic expansion
for $\lm_1^\aa(\beta)$. However, we will not elaborate on this point
here.

\subsection*{Organisation of the paper}
In Section~\ref{sec:BS} we recall the standard formulation of the
Birman-Schwin\-ger principle for the Hamiltonian $\Op$ and employ it
to obtain a useful lower bound on $\lm_1^\aa(\beta)$. Furthermore,
we  derive  a perturbative reformulation of the  Birman-Schwinger
principle and expand the new  Birman-Schwinger condition with
respect to the small parameter $\beta$. In Section~\ref{sec:proof}
we prove our main result, formulated in Theorem~\ref{thm:main}. We
conclude the paper by Section~\ref{sec:dis} containing a discussion
on possible generalizations of the obtained results.

\section{Birman-Schwinger principle}\label{sec:BS}

\subsection{Standard formulation}
Birman-Schwinger principle (BS-principle in what follows) is a
powerful tool for the spectral analysis of Schr\"odinger operators.
Its generalization, which covers $\delta$-interactions supported on
hypersurfaces, is derived in~\cite{BEKS}; see also~\cite{BLL13, B95,
Po01} for some refinements.

In what follows, let $\lm < 0$ and set $\kp := \sqrt{-\lm}$. Green's
function corresponding to the differential expression $-\Delta +
\kp^2$ in $\R^3$ takes the following well-known form
\begin{equation*} \label{freeG3}
    G_\kp(x\! - \!y) = \frac{\e^{-\kp|x-y|}}{4\pi | x \!- \! y|}.
\end{equation*}
Let the surface $\Gamma = \Gamma_\beta(f)\subset\R^3$ be as
in~\eqref{eq:Gamma}. Parametrizing $\Gamma$ by the mapping
\begin{equation}\label{eq:rbeta}
    r_\beta\colon \R^2\arr \R^3,\qquad
    r_\beta(x) := \left(x, \beta f(x) \right),
\end{equation}
we can naturally express the surface measure on $\Gamma$
through the Lebesgue measure on $\dR^2$
via the relation $\dd\s(x) = g_\beta(x) \dd x$, where
the Jacobian $g_\beta$ is explicitly given by
\begin{equation*}\label{eq:Jacobian}
    g_\beta(x)
    =
    \left (1 + \beta^2 |\nabla f(x)|^2\right )^{1/2}.
\end{equation*}
Next we introduce the \emph{weakly singular integral operator}
$\Qb(\kp)\colon L^2(\R^2 )\to L^2(\R^2)$, $\kp > 0$, acting as
\begin{equation}\label{def:SL}
    \left(\Qb(\kp) \psi\right ) (x)
    :=
    \int_{\R^2}
    g_\beta(x)^{1/2}
        G_\kp\left (r_\beta(x)-  r_\beta(y)\right )
    g_\beta(y)^{1/2} \psi(y) \dd y.
\end{equation}
Note that the linear mapping $\sfJ_\beta \colon L^2(\Gamma) \arr
L^2(\dR^2)$, $(\sfJ_\beta \psi)(x) =
g_\beta(x)^{1/2}\psi(r_\beta(x))$, is an isometric isomorphism and
it is not difficult to check that $\Qb(\kp) = \sfJ_\beta
R_{mm}(\ii\kp)\sfJ_\beta^{-1}$, where the operator
$R_{mm}(\ii\kp)\colon L^2(\Gamma)\arr L^2(\Gamma)$ is defined as
\begin{equation*}\label{key}
    (R_{mm}(\ii\kp) \psi)(x)
    := \int_{\Gamma} G_\kp(x-y) \psi(y)\dd \s(y).
\end{equation*}
In fact, $R_{mm}(\ii\kp)$ is the Birman-Schwinger operator
introduced in~\cite[Sec. 2]{BEKS}, see also~\cite{B95}. In view of
this identification, we get from~\cite{B95} that $\Qb(\kp)$ is a
bounded, self-adjoint, non-negative operator in $L^2(\dR^2)$.
Next theorem contains a BS-principle for the Schr\"o\-dinger
operator $\Op$ in Definition~\ref{def:Op}. We remark that while this
formulation of the BS-principle is not the same as in \cite[Lem.
2.3\,(iv)]{BEKS} and~\cite[Lem. 1]{B95}, it can be easily derived
from those claims using the identity $\Qb(\kp) = \sfJ_\beta
R_{mm}(\ii \kp)\sfJ_\beta^{-1}$.
\begin{thm}\label{thm:BS}
    Let the self-adjoint operator $\Op$
    be as in Definition~\ref{def:Op}
    and the operator-valued function $\R_+ \ni \kp  \mapsto
    \Qb(\kp)$ be as in~\eqref{def:SL}.
    Then it holds that
    \[
        \forall\, \kp > 0,
        \qquad
        \dim\ker\big(\Op + \kp^2\big)
        =
        \dim\ker\big(\sfI -\aa \Qb(\kp)\big).
    \]
\end{thm}%
In the following lemma we recall the properties of $\sfQ_0(\kp)$
(\ie~for $\beta = 0$). Since these properties are easy
to prove and difficult to find in the literature
we provide a short argument.
\begin{lem} \label{lem:fourier}
The operator $\sfQ_0(\kp)$ is unitarily equivalent via the Fourier
transformation to the multiplication operator in  $\R^2$ with the
function
    \[
        \R^2\ni p \mapsto \frac{1}{2\sqrt{|p|^2 + \kp^2}}.
    \]
    In particular, $\s(\sfQ_0(\kp)) = [0,\frac{1}{2\kp}]$
    and the operator-valued function $\dR_+\ni\kp \mapsto
    \sfQ_0(\kp)$ is real analytic.
\end{lem}
\begin{proof}
    Recall that for the Fourier transform of the convolution of $\psi_1,\psi_2\in L^2(\dR^2)$ we have
    $\cF(\psi_1\star \psi_2) = \wh \psi_1\wh \psi_2$.
    Using this formula and the fact that $p\mapsto\frac{1}{2\sqrt{|p|^2 + \kp^2}}$
    is the Fourier transform of $\dR^2\ni x\mapsto \frac{e^{-\kp|x|}}{4\pi|x|}$
    we get for any $\psi \in L^2(\dR^2)$
    \begin{equation*}\label{key}
    \begin{split}
        \sfQ_0(\kp)\psi
        & =
        \cF^{-1}\cF\int_{\dR^2} G_\kp( \cdot -y) \psi(y) \dd y\\
        & =
        \cF^{-1}
        \left (\cF
        \left (\frac{e^{-\kp|\cdot|}}{4\pi|\cdot|}\right ) \wh\psi \right )
         =
        \cF^{-1}
        \left ( \frac{\wh\psi}{2\sqrt{|\cdot|^2 + \kp^2}}\right ),
    \end{split}
    \end{equation*}
    and the main claim of the lemma immediately follows. The analyticity of
    $\dR_+\ni\kp\mapsto \sfQ_0(\kp)$ is a consequence
    of the same property of
    the operator-valued function
    of multiplication by
    $\dR_+\ni\kp\mapsto\frac{1}{2\sqrt{|p|^2 + \kp^2}}$.
    Moreover, we have
    \[
        \s(\sfQ_0(\kp)) = \ov{\left \{\lm \in\dR \colon
        \lm =  \tfrac{1}{2\sqrt{|p|^2 + \kp^2}}~\text{for }\,
        p\in\dR^2\right \}} = \left [0,\tfrac{1}{2\kp}\right ].
        \qedhere
    \]
\end{proof}
By means of the BS-principle in Theorem~\ref{thm:BS} we obtain a
useful lower bound on the lowest spectral
point $\lm_1^\aa(\beta)$ of $\Op$.
\begin{prop}\label{prop:lower_bound}
    Let $\lm_1^\aa(\beta)$ be the lowest spectral point
    of the self-adjoint operator $\Op$ introduced in Definition~\ref{def:Op}.
    Then the following lower bound
    \begin{equation*}\label{key}
        \lm_1^\aa(\beta) \ge -\frac{\aa^2}{4}
        \left (1+\beta^2 L_f^2\right )
    \end{equation*}
    holds for all $\aa,\beta > 0$. In particular,
    $\lm_1^\aa(\beta)\arr -\frac14\aa^2-$ as $\beta \arr 0+$.
\end{prop}
\begin{proof}
    In view of Proposition~\ref{thm:known}\,(i) we clearly have
    $\lm_1^\aa(\beta) \le -\frac14\aa^2$ and if $\lm_1^\aa(\beta) < -\frac14\aa^2$, then necessarily $\lm_1^\aa(\beta)\in\sd(\Op)$
    holds.  Applying the Schur test~\cite[Lem. 0.32]{Te} for the operator $\aa\Qb(\kappa)$
    we get, using monotonicity of $G_\kp(\cdot)$
    in combination with the inequalities $|r_\beta(x) - r_\beta(y)| \ge |x-y|$ and $|\nabla f| \le L_f$,
    the following bound
    \begin{equation*}\label{key}
    \begin{split}
        \|\aa\Qb(\kp)\|
        & \le \aa(1+\beta^2L_f^2)^{1/2}
        \sup_{x\in\dR^2}
        \int_{\dR^2} G_{\kp}(r_\beta(x) - r_{\beta}(y)) \dd y\\
        & \le
        \aa
        (1+\beta^2 L_f^2)^{1/2}
        \sup_{x\in\dR^2}
        \int_{\dR^2} \frac{\e^{-\kp|x - y|}}{4\pi|x-y|} \dd y\\
        &
        =
        \aa
        (1+\beta^2 L_f^2)^{1/2}
        \int_{\dR^2} \frac{\e^{-\kp|y|}}{4\pi|y|} \dd y \\
        & =
        \aa(1+\beta^2 L_f^2)^{1/2}\frac12
        \int_0^\infty e^{-\kp r} \dd r =
        \frac{\aa}{2\kp}(1+\beta^2 L_f^2)^{1/2}.
    \end{split}
    \end{equation*}
    Consequently, for $\kp > \frac{\aa}{2}(1+\beta^2 L_f^2)^{1/2}$
    holds $\|\aa\Qb(\kp)\| < 1$ and by the BS-principle in
    Theorem~\ref{thm:BS} we get $-\kp^2\notin\sd(\Op)$.
    Finally, we conclude that
    \[
        \lm_1^\aa(\beta) \ge -\frac{\aa^2}{4}(1+\beta^2 L_f^2).
        \qedhere
    \]
\end{proof}

\subsection{Perturbative reformulation}
In our considerations it is convenient to deal with a perturbative
reformulation of the BS-principle. This technique has already been
successfully applied in~\cite{BFKLR17, CK11, EK02, EK08, EK15} for
the case of interactions supported on curves. To this aim, for $\kp
\ge \frac12\aa$ we set $\delta := \sqrt{\kp^2 -\frac14\aa^2}$ and
define the operator-valued function
\begin{equation*}\label{eq:D}
    \Db(\delta) := \Qb(\kp) - \sfQ_0(\kp),
\end{equation*}
 which is real analytic in $\delta\in(0,\infty)$ and in $\beta \in (0,\infty)$;
 \cf~\cite[\S VII.1.1]{Kato} and the explicit expression for the integral kernel in~\eqref{def:SL}.
%
%
Next, for $\kp > \frac12\aa$ we define
\begin{equation}\label{eq:B}
    \sfB_\aa (\delta) := \big(\sfI - \aa\sfQ_0(\kp)\big)^{-1},
\end{equation}
where existence and boundedness of the inverse of
$\sfI - \aa\sfQ_0(\kp)$ are guaranteed by Lemma~\ref{lem:fourier}.

The above auxiliary operators satisfy
%
\[
    \dim\ker(\sfI -  \aa\Qb(\kp) )
    =
    \dim\ker\big(\sfI - \aa\sfQ_0(\kappa) -
    \aa\Db(\delta)\big)
    =
    \dim\ker\big(\sfI -  \aa\sfB_\aa(\delta)\Db(\delta)\big).
\]
Thus, the BS-principle formulated  in Theorem~\ref{thm:BS}  yields
\begin{equation}\label{eq:reformulation}
    \forall\, \kp > \frac{\aa}{2},
    \qquad
    \dim\ker\big(\Op + \kp^2\big)
    =
    \dim\ker\big(\sfI -  \aa\sfB_\aa(\delta)\Db(\delta)\big).
\end{equation}
In the next lemma we collect the properties of the operator family $\sfB_\aa(\delta)$.
In the following, we denote by $-\Delta_{\R^2}$  the usual self-adjoint free Laplacian in $L^2(\R^2)$, whose resolvent
is abbreviated by $\sfR(z) := (-\Delta_{\dR^2} + z)^{-1}$ for $z >0$.
\begin{lem} \label{le-decompoB}
    The operator $\sfB_\aa(\delta)$, $\delta >0$, in~\eqref{eq:B} admits the representation:
    \begin{equation} \label{eq-decompoBa}
        \sfB_\aa(\delta) = \frac{\aa^2}{2} \sfR(\delta^2) + \sfN_\aa(\delta)
    \end{equation}
    with
    \begin{equation*}\label{eq:N}
        \sfN_\aa(\delta) := 1 +
        \aa \sfR\big(\delta^2 + \tfrac14\aa^2\big)^{1/2}\left(
        \aa \sfR\big(\delta^2 + \tfrac14\aa^2\big)^{1/2} + 2
        \right )^{-1},\qquad \delta\ge 0.
    \end{equation*}
    Moreover, the operator-valued function $\sfN_\aa(\delta)$
    satisfies the following properties.
    \begin{myenum}
        \item
        The estimate
        $\|\sfN_\aa(\delta)\| \le \frac32$ is valid for all $\delta \ge 0$.
        \item
        The convergence $\sfN_\aa(\delta)\arr\sfN_\aa(0)$ holds
        in the operator norm as $\delta\arr 0+$.
        \item
        $(0,\infty)\ni\delta\mapsto \sfN_\aa(\delta)$ is real analytic.
        \item
        The estimate\footnote{Here and in the following we define
            the derivative of an operator-valued function
            $\dR_+\ni\delta\mapsto \sfA(\delta)$ as
            the limit in the operator-norm
            of the fraction
            $\frac{\sfA(\dl') -\sfA(\dl)}
            {\dl' - \dl}$ as $\dl'\arr\dl$.}
        $\|\p_\delta\sfN_\aa(\delta)\| \le \frac{\delta}{\aa^2}$
        is valid for all $\delta \ge 0$.
    \end{myenum}
    In particular, representation~\eqref{eq-decompoBa}
    yields real analyticity of $\sfB_\aa(\dl)$
    with respect to $\delta \in (0,\infty)$.
\end{lem}

\begin{proof}
    By Lemma~\ref{lem:fourier}, the operator $\sfB_\aa(\delta)$ is unitarily equivalent (via the Fourier transformation) to the operator of multiplication with the function
    \[
        f_{\aa,\dl}(p)
        :=
        \bigg(1 - \frac{\aa}{2\tau_{\aa,\dl} (p)}\bigg)^{-1},
    \]
    where
    $\tau_{\aa,\dl} (p)
    :=
    \sqrt{|p|^2 +\dl^2 +\frac14\aa^2}$.
    Note that the function $f_{\aa,\delta}$ can be decomposed as
    $f_{\aa,\dl}(p) = m_{\aa,\dl}(p) + n_{\aa,\dl}(p)$ with
    \[
        m_{\aa,\dl}(p) := \frac{\aa^2}{2(|p|^2 + \delta^2)}
        \qquad\text{and}\qquad
        n_{\aa,\dl}(p) := 1 + \frac{\aa}{2\tau_{\aa,\dl} (p) + \aa}\,.
    \]
    Observe that we have
    \begin{equation}\label{eq:f2}
        n_{\aa,\dl}(p) \le  n_{\aa,\dl}(0)
        =1 + \frac{\aa}{2\tau_{\aa,\dl}(0)+\aa}
        \le
        1 + \frac{\aa}{2\tau_{\aa,0}(0)+\aa}
        =  1 + \frac{1}{2} = \frac{3}{2}.
    \end{equation}
    Clearly, the operators of multiplication
    with $m_{\aa,\dl}$ and with
    $n_{\aa,\dl}$ are unitarily equivalent
    via the inverse Fourier transformation to
    $\frac{\aa^2}{2}\sfR(\delta^2)$ and to $\sfN_\aa(\delta)$,
    respectively. Hence, the decomposition~\eqref{eq-decompoBa} is valid. In particular,
    an  upper bound determined in (i) holds, thanks to~\eqref{eq:f2}.

    The estimate
    \[
    \begin{split}
        \left \|\sfN_\aa(\delta) - \sfN_\aa(0)\right \|
        & =
        \sup_{p\in\dR^2} \left |
        \frac{\aa}{2\tau_{\aa,\dl} (p) + \aa}
        -
        \frac{\aa}{2\tau_{\aa,0} (p) + \aa}
        \right |\\
        & \le
        \frac{1}{2\aa}
        \sup_{p\in\dR^2} \left |
        \tau_{\aa,0} (p) - \tau_{\aa,\dl} (p)\right | \le
        \frac{\delta^2}{2\aa}
        \sup_{p\in\dR^2}
        \frac{1}{\tau_{\aa,0} (p) + \tau_{\aa,\dl} (p)}
        \le
        \frac{\delta^2}{2\aa^2},
    \end{split}
    \]
    implies the convergence in (ii).
    Analyticity of $(0,\infty)\ni \delta\mapsto \sfR(\frac14\aa^2 + \delta^2)$ yields the claim of~(iii).

    Define $\p_\dl \sfN_\aa(\delta)\colon L^2(\dR^2)\arr L^2(\dR^2)$
    as the operator being unitarily equivalent via the Fourier transformation to the multiplication with
    the function
    \begin{equation*}\label{key}
        \p_\dl n_{\aa,\dl}(p) =
        -\frac{2\aa\dl}{\left (2\tau_{\aa,\dl}(p)+ \aa\right )^2\tau_{\aa,\dl}(p)}.
    \end{equation*}
    Next, we show that the operator
    $\p_\dl \sfN_\aa(\dl )$ defined as above
    satisfies
    \begin{equation}\label{eq-1N}
        \lim _{\dl'\to \dl}
        \left\|
        \frac{\sfN_\aa (\dl ')- \sfN_\aa (\dl)}{\dl '-\dl } - \p_\dl \sfN_\aa (\dl )
        \right\| = 0\,.
    \end{equation}
    Applying the mean-value theorem we obtain
    \[
        \left|
        \frac{n_{\aa,\dl'} (p)-n_{\aa,\dl} (p)}
        {\dl'-\dl } -
            \p_\dl n_{\aa,\dl} (p) \right|
        \leq
        \left|
        \p_{\dl\dl}^2 n_{\aa,\dl_\star} (p)(\dl '-\dl)\right|,
    \]
    where $\delta_\star\in (0,\infty)$ lies between $\delta$ and $\delta'$.
    A straightforward calculations shows
    \[
        \p_{\dl\dl}^2 n_{\aa,\dl} (p)
        =
        -\frac{2\aa }{ (2\tau_{\aa,\dl} (p)+\aa )^2 \tau_{\aa,\dl} (p)}
        +
        \frac{8\aa \dl^2 }{(2\tau_{\aa,\dl} (p) +\aa )^3 \tau_{\aa,\dl} (p)^2}
        +
        \frac{2\aa \dl^2}{(2\tau_{\aa,\dl} (p)+\aa )^2 \tau_{\aa,\dl} (p)^3},
    \]
    which implies
    \[
        \left| \p_{\dl\dl}^2 n_{\aa,\dl} (p)\right|
        \leq \frac{3}{\aa^2},
    \]
    and hence
    \[
        \sup_{p\in \R^2}
        \left|
        \frac{n_{\aa,\dl'} (p)- n_{\aa,\dl} (p)}{\dl'-\dl } - \p_\dl n_{\aa,\dl} (p ) \right|
        \leq
        \frac{3}{\aa^2}\left| \dl '-\dl \right|.
    \]
    This completes the verification of~\eqref{eq-1N}.

    Finally, we get
    \begin{equation*}\label{key}
    \begin{split}
        \|\p_\dl \sfN_\aa(\delta)\|
        & =
        \sup_{p\in\dR^2}
        \frac{2\aa\dl}{\left (2\tau_{\aa,\dl}(p)+ \aa\right )^2\tau_{\aa,\dl}(p)}   \\
        & =
        \frac{2\aa\dl}{\left (2\tau_{\aa,\dl}(0)+ \aa\right )^2\tau_{\aa,\dl}(0)}
        \le
        \frac{2\aa\dl}{\left (2\tau_{\aa,0}(0)+ \aa\right )^2\tau_{\aa,0}(0)} =  \frac{\delta}{\aa^2},
    \end{split}
    \end{equation*}
    which settles the claim of (iv).
\end{proof}
In what follows, we identify $x\in\dR^2$ with $(x,0)\in\dR^3$. For a
given measurable $V\colon \dR^2\arr \dR_+$ we introduce the integral kernels
\begin{subequations}\label{eq:kernels}
\begin{align}
    \wDb(\delta)(x,y)         & := V(x)\Db(\delta)(x,y) V(y),
    \label{eq:kernel1}\\
    \sfD^{(1)}_V(\delta)(x,y) & := V(x)G_\kp(x-y)V(y)E(x,y),
    \label{eq:kernel2}
\end{align}
\end{subequations}
where
\begin{equation*}\label{eq:E}
    E(x,y) :=
    \frac{|\nabla f(x)|^2 + |\nabla f(y)|^2}{4}
    -
    \frac{|f(x) - f(y)|^2(\kp|x-y|+ 1)}{2|x-y|^2}.
\end{equation*}
Furthermore, we work out a representation for
the operator-valued function $\wDb(\delta)$
associated with the kernel in~\eqref{eq:kernel1}
under certain limitation on the growth of $V$.
\begin{prop}\label{prop:expansion}
    Let a measurable $V\colon\dR^2\arr \dR_+$ satisfy  $V(x) \le c\exp\left (\frac{\aa}{4}|x|\right )$ for all $x\in\dR^2$ with some constant $c > 0$.
    Let the integral kernels $\wDb(\delta)$ and $\sfD_V^{(1)}(\delta)$, $\delta\in [0,1]$, $\beta \in (0,1]$, be as in~\eqref{eq:kernels}.
    Then there exist constants $C_j = C_j(\aa,f,c) > 0$, $j = 1,2,3$,
    such that the following claims hold.
    \begin{myenum}
        \item
        For all $x,y\in\dR^2$, the pointwise bound
        \begin{equation}\label{eq:pntws_bnd}
            |\sfD^{(1)}_V(\delta)(x,y)|\le
            C_1 G_{\frac{\aa}{4}}(x-y)\left [1 + \frac12 \kp |x-y|\right ],
        \end{equation}
        holds, the kernel $\sfD^{(1)}_V(\delta)(x,y)$
        defines the self-adjoint operator $\sfD^{(1)}_V(\delta)\in\cB(L^2(\dR^2))$,
        and, in addition, $\|\sfD^{(1)}_V(\delta)\| \le C_2$.

        \item
        For all $x,y\in\dR^2$, the decomposition
        \begin{equation}\label{eq:representation_kernel}
            \wDb(\delta)(x,y)
            =
            \sfD_V^{(1)}(\delta)(x,y)\beta^2 + \wDb^{(2)}(\delta)(x,y)\beta^4
        \end{equation}
        holds, the kernel $\wDb^{(2)}(\delta)(x,y)$ defines the self-adjoint operator $\wDb^{(2)}(\delta)\in\cB(L^2(\dR^2))$,
        and, in addition,
        $\|\wDb^{(2)}(\delta )\| \le  C_3$.
    \end{myenum}
    In particular, the kernel $\wDb(\delta)(x,y)$ induces the self-adjoint operator
    \begin{equation*} \label{eq:representation}
        \wDb(\delta)
        =
        \sfD^{(1)}_V(\delta )\beta^2 + \wDb^{(2)}(\delta )\beta^4\in\cB(L^2(\dR^2))\,.
    \end{equation*}
\end{prop}
\begin{proof}
    \noindent{\rm (i)}\,
    Recall that $\cS = \supp f$ and let $\cB_R\subset\dR^2$ be an open ball of the radius $R > 0$ centred at the origin such
    that the inclusion $\cS\subset\cB_R$ holds.
    The subset
    of $\dR^2\times\dR^2$, where the factor $E(x,y)$ in the expression~\eqref{eq:kernel2} for $\sfD^{(1)}_V(\delta)(x,y)$
    is not equal to zero can be covered by two (intersecting) sets
    \begin{equation}\label{eq:UV}
        \cU := \cB_R\times\cB_R\quad\text{and}\quad
        \cV := (\cS\times\cB_R^{\rm c}) \cup (\cB_R^{\rm c}\times\cS),
    \end{equation}
    where $\cB_R^{\rm c} := \dR^2\setminus \ov{\cB_R}$. Applying the bound
        $\frac14|x-y| > \frac14(|x| + |y|) - \frac12 R$ (valid for all $(x,y)\in \cU\cup \cV$)  we get
    \begin{equation}\label{eq:VGV2}
        V(x)G_\kp(x-y)V(y)
        \le
        c^2\exp\left (\frac{\aa R}{2}\right ) G_{\frac{\aa}{4}}(x-y),
        \qquad \forall (x,y)\in \cU\cup\cV,
    \end{equation}
    where we also used  monotonicity of Green's function with respect to $\kp$.

    Employing the inequality
    $|\nabla f| \le L_f$ we can pointwise estimate the factor
    $E$ by
    \begin{equation}\label{eq:E}
        |E(x,y)| \le
        L_f^2\left[ 1 + \frac12  \kp |x-y|\right ].
    \end{equation}
    Combining~\eqref{eq:VGV2} and \eqref{eq:E} we get   
    the bound~\eqref{eq:pntws_bnd} with
    \begin{equation}\label{eq:C1}
        C_1 := c^2L^2_f\exp\left (\frac{\aa R}{2}
        \right ).
    \end{equation}
    Taking into account that the integral
    kernel of $\sfD^{(1)}_V(\delta)$ is symmetric,
    we obtain from~\eqref{eq:pntws_bnd} using the Schur test that
    \begin{equation*}\label{eq:wtD1}
    \begin{split}
        \|\sfD^{(1)}_V(\delta)\| &
        \le
        C_1
        \int_{\dR^2} G_{\frac{\aa}{4}}(x) \dd x
        +
        C_1\frac{\kp}{2}
        \int_{\dR^2} G_{\frac{\aa}{4}}(x)|x| \dd x
        \\
        & =
        C_1\frac{1}{2}
        \int_0^\infty e^{-\frac{\aa}{4} r} \dd r +
        C_1\frac{\kp}{4}
        \int_0^\infty e^{-\frac{\aa}{4} r} r \dd r\\
        & =
        C_1
        \left (\frac{2}{\aa} + \frac{4\kp}{\aa^2}\right)
        \le C_1
        \left ( \frac{4}{\aa} + \frac{4}{\aa^2}\right ) =: C_2,
    \end{split}
    \end{equation*}
    in the last step of the above estimates we employed that $\kp = \sqrt{\frac14\aa^2 + \dl^2}
    \le \frac12\aa + \dl \le \frac12\aa + 1$ for all $\dl \in [0,1]$.
    Thus, the kernel $\sfD_V^{(1)}(\delta)(x,y)$ defines
    the operator $\sfD_V^{(1)}(\delta)\in\cB(L^2(\dR^2))$.
    Self-adjointness of $\sfD_V^{(1)}(\delta)$
    is a consequence of the identity
    $\sfD_V^{(1)}(\delta)(x,y) = \sfD_V^{(1)}(\delta)(y,x)$.

    \smallskip

    \noindent {\rm  (ii)}\,
    For $x,y\in\R^2$, we introduce
    $\rho_\beta(x,y) :=
        \left | r_\beta (x)- r _\beta (y)\right |^2$,
    where the mapping $r_\beta\colon \dR^2\arr\dR^3$ is as
    in~\eqref{eq:rbeta}.
    A simple computation yields
    \begin{equation*}
        \rho_\beta(x,y)
        =
        |x-y|^2 + |f(x) - f(y)|^2 \beta^2.
    \end{equation*}
    Furthermore, we define the function $F\colon\R_+\arr \R_+$
    by $F(s) := \frac{\e^{-\kp\sqrt{s}}}{4\pi \sqrt{s}}$
    and compute its first and second derivatives
    \[
    \begin{split}
        F'(s) & =   -\frac{\e^{-\kp\sqrt{s}} \left(\kp s^{1/2} + 1\right)}{8\pi s^{3/2}} = -F(s)\frac{\kp s^{1/2}+1}{2s},\\
        F''(s)&  =
        \frac{\e^{-\kp \sqrt{s}}\big[\kp^2 s + 3 \kp s^{1/2} + 3\big]}{16\pi s^{5/2}}= F(s) \frac{\kp^2 s + 3 \kp s^{1/2} + 3}{4s^2} .
    \end{split}
    \]
    Taylor expansion of $F(\cdot)$ in the vicinity of $s \in (0,\infty)$
    with the remainder in the Lagrange form reads as
    follows
    \[
        F(t) = F(s) + F'(s)(t-s) +
        F''(s + \tt\cdot(t-s))\frac{(t-s)^2}{2}\,,
        \qquad
        \tt= \tt(s,t) \in (0,1).
    \]
    For $x,y\in\dR^2$ we define an auxiliary function
    $\mu\colon\dR^2\times\dR^2\arr \dR_+$ by
    \[
        \mu(x,y) :=
        \rho_0(x,y) +
        \tt(\rho_0(x,y),\rho_\beta(x,y))
        |f(x)-f(y)|^2\beta^2.
    \]
    For the sake of brevity we denote
    \[
        H(x,y)  :=
        \frac{|f(x) - f(y)|^2(\kp|x-y| + 1)}{2|x-y|^2},
        \quad
        K(x,y)  :=  \frac{|\nabla f(x)|^2 + |\nabla f(y)|^2 }{4},
    \]
    and
    \begin{subequations}
        \begin{align}
        K_1(x,y) &
        :=
        \left(K(x,y) -  H(x,y)\right )\beta^2 = E(x,y)\beta^2,\label{eq:K1} \\
        K_2(x,y) &
        :=
        \left (g_\beta(x)g_\beta(y)\right )^{1/2} - 1 - \beta^2 K(x,y),
        \label{eq:K2}\\
        K_3(x,y) &
        :=
        H(x,y)
        \left (1 -
            \left (g_\beta(x)g_\beta(y)\right )^{1/2} \right )\beta^2,\label{eq:K3}\\
        K_4(x,y) &
        :=
        g_\beta(x)^{1/2} F''(\mu(x,y)) g_\beta(y)^{1/2}
        \frac{|f(x) - f(y)|^4}{2}\beta^4.\label{eq:K4}
    \end{align}
    \end{subequations}
    Dependence of the above kernels on $\beta$
    and $f$ is not indicated in the notations as no confusion can arise.
    Thus, the integral kernel $\wDb(\delta)(x,y)$ can be decomposed as
    \[
    \begin{split}
        \wDb(\delta)(x,y)
        & =
        V(x)\left [\Qb(\delta)(x,y) - \sfQ_0(\delta)(x,y) \right ] V(y)\\
        & =
        V(x)\left (G_\kp(x-y)\sum_{j=1}^3 K_j(x,y) +
        K_4(x,y)\right )V(y).
    \end{split}
    \]
    Hence, the  expansion~\eqref{eq:representation_kernel} holds with the integral kernel of $\wDb^{(2)}(\delta)$ given by
    \begin{equation}\label{eq:wDb}
        \wDb^{(2)}(\delta)(x,y)\! =
        \frac{V(x)V(y)}{\beta^4}
        \big [G_\kp(x-y)
        \left (K_2(x,y) + K_3(x,y)\right )
        + K_4(x,y)\big ].
    \end{equation}
    %
    With the aid of the definitions~\eqref{eq:K2},~\eqref{eq:K3}
    for the kernels
    $K_j(\cdot,\cdot)$, $j=2,3$,
    one obtains using $\beta \in (0,1]$
    and $|\nabla f| \le L_f$ that
    %
    \begin{equation}\label{eq:estimates_K23}
        |K_2(x,y)|
        \le
        C_{2,f}\beta^4
        \qquad\text{and}\qquad
        |K_3(x,y)|  \le
        C_{3,f}\left (\kp|x-y| +1\right)\beta^4,
    \end{equation}
    with some constants $C_{2,f}, C_{3,f} > 0$.
    Taking into account that $F''$ is a decreasing  positive function and using that $\beta \in (0, 1]$
    we estimate $K_4$
    in~\eqref{eq:K4} as
    \begin{equation}\label{eq:estimate_K4}
        |K_4(x,y)|
        \le
        C_{4,f}
        G_{\kp}(x-y) \left (\kp^2|x-y|^2 + 3\kp|x-y| + 3\right )\beta^4,
    \end{equation}
    with some constant $C_{4,f} > 0$.
    Finally, combining the estimates~\eqref{eq:VGV2},~\eqref{eq:estimates_K23},
    ~\eqref{eq:estimate_K4}, and the expression
    for $\wDb^{(2)}(\delta)(\cdot,\cdot)$ in~\eqref{eq:wDb} we end up with
    \[
        |\wDb^{(2)}(\delta)(x,y)|
        \le
        C_1C_3' G_{\frac{\aa}{4}}(x-y)
        \left [5+ 4\kp|x-y| +  \kp^2|x-y|^2\right ],
    \]
    where $C_3' :=\,{\rm max}\{C_{2,f},
    C_{3,f},C_{4,f}\}$
    and $C_1$ is as in~\eqref{eq:C1}.
    Applying the Schur test once again we get
    \begin{equation*}\label{eq:wtD2}
    \begin{split}
        \|\wDb^{(2)}(\delta)\| &
        \le
         C_1C_3'
        \int_{\dR^2} G_{\frac{\aa}{4}}(x)
        \left [5 + 4\kp|x| +  \kp^2|x|^2\right ] \dd x\\
        & =
        \frac12
        C_1 C_3'
        \int_0^\infty e^{-\frac{\aa}{4} r}\left [5 + 4\kp r +  \kp^2r^2\right ] \dd r
        =
        C_1 C_3'
        \left [ \frac{10}{\aa} +
        \frac{32}{\aa^2}\kp + \frac{64}{\aa^3}\kp^2\right ]\\
        &
        \le
        C_1 C_3'
        \left [ \frac{10}{\aa} +
        \frac{32}{\aa^2}\left (\frac{\aa}{2}+1\right ) + \frac{64}{\aa^3}\left  (\frac{\aa^2}{4}+1\right)\right ] =: C_3,
    \end{split}
    \end{equation*}
    where we used the bounds $\kp^2 \le \frac14\aa^2 + 1$ and $\kp \le \frac12\aa + 1$.
    Thus, we have shown $\wDb^{(2)}(\dl)\in \cB(L^2(\dR^2))$. Self-adjointness of $\wDb^{(2)}(\dl)$ follows from
    $\wDb^{(2)}(\dl)(x,y) = \wDb^{(2)}(\dl)(y,x)$.
\end{proof}
In the next proposition we show real analyticity
of $\wDb(\delta)$ with respect to $\delta$ and $\beta$. Furthermore, we estimate the norm of
$\p_\delta \wDb(\delta)$.
\begin{prop}\label{prop:analytic}
    Let the assumptions be as in Proposition~\ref{prop:expansion}.
    Then the following claims hold.
    \begin{myenum}
        \item The operator-valued function $(0,1)^2\ni(\delta,\beta)\mapsto \wDb(\dl)$ is real analytic in both arguments.
        \item $\|\p_\dl \wDb(\dl)\| = \cO_{\rm u}(1)$
        as $\beta \arr 0+$.
    \end{myenum}
\end{prop}
\begin{proof}
    \noindent (i) Combining~\cite[Thm. III 3.12]{Kato} (and the discussion in~\cite{Kato} following it) with the claims of Proposition~\ref{prop:expansion} we conclude that it suffices to check real analyticity with respect to $\delta,\beta \in (0,1)$ of
    the scalar-valued functions
    \begin{equation*}\label{key}
            (0,1)^2\ni(\delta,\beta)\mapsto \left (\wDb(\dl)h_1,h_2\right )_{L^2(\dR^2)},
    \end{equation*}
    where $h_1,h_2\in C^\infty_0(\dR^2)$.
    The latter follows from real analyticity of $(0,1)^2\ni(\delta,\beta)\mapsto\Db(\dl)$ in $\delta$ and $\beta$, because the function $V$ is locally bounded.

    \noindent (ii)
    Differentiating the integral kernel $\wDb(\delta)(x,y)$ with respect
    to $\dl$ we find
    \begin{equation*}\label{key}
    \begin{split}
    \p_\dl\wDb(\delta)(x,y)
    & =
    V(x)\p_\dl(\Db(\delta)(x,y)) V(y)\\
    & =
    \frac{\delta}{\kp}
    V(x)
    \left [g_\beta(x)^{1/2}
    \p_\kp \big(G_\kp(r_\beta(x) - r_\beta(y))\big)
    g_\beta(y)^{1/2} - \p_\kp \big(G_\kp(x - y)\big)\right ] V(y)\\
    & =
    \frac{\delta}{4\pi\kp}
    V(x)
    \left [e^{-\kp|x - y|} - g_\beta(x)^{1/2}
    e^{-\kp|r_\beta(x) - r_\beta(y)|}
    g_\beta(y)^{1/2} \right ] V(y).
    \end{split}
    \end{equation*}
    Next, we show that the integral operator
    $\p_\dl \wDb(\dl)\colon L^2(\dR^2)\arr L^2(\dR^2)$ associated with the above
    kernel satisfies
    \begin{equation}\label{eq-2D}
        \lim _{\dl'\to \dl }
        \left\|
        \frac{\sfD_{\beta, V} (\dl')- \sfD_{\beta, V} (\dl )}{\dl '-\dl } - \p_ \dl \sfD_{\beta, V} (\dl ) \right\| = 0.
    \end{equation}
    Applying the mean-value theorem for the
    integral kernels, we get
    \[
        \left|
        \frac{\sfD_{\beta , V}(\dl ')(x,y)-
            \sfD_{\beta , V}(\dl )(x,y)
        }{\dl'-\dl }-
            \p_\dl\sfD_{\beta , V}(\dl)(x,y)
        \right|
        \leq
        |\p_{\dl\dl}^2\sfD_{\beta,V}(\dl_\star )(x,y) (\dl'-\dl )|\,,
    \]
    where $\delta_\star$ lies between $\delta '$ and $\delta$.
    Standard calculations yield
    \[
    \begin{split}
        \p^2_{\dl\dl}\sfD_{\beta,V}(\delta )(x,y)
        & =
        \frac{1}{4\pi} V(x)   \left(
        \big(g_\beta(x)g_\beta(y)\big)^{1/2}
        \left(\frac{\dl ^2}{\kp^3}-\frac{1}{\kp }+\frac{\dl^2|r_\beta (x)-r_\beta (y)|}{\kp^2 }\right) e ^{-\kp |r_\beta (x)-r_\beta (y)|}
        \right.
        \\
        &\qquad\qquad\qquad\qquad\left. - \left( \frac{\dl^2}{\kp^3} -\frac{1}{\kp}+\frac{\dl^2|x-y| }{\kp^2} \right) e ^{-\kp |x-y|}
        \right)V(y).
    \end{split}
    \]
    Using the inequality $|r_\beta (x)-r_\beta (y)|\leq 2|x-y|$ which holds  for $L_f \beta \leq \sqrt{3}$ together
    with the estimates $\|g_\beta\|_\infty \le 1+ L_f$ and $\kp >\aa /2$ we get  
    \begin{equation*}
        |\p^2_{\dl\dl}\sfD_{\beta, V}(\dl )(x,y) |
        \leq
        \frac{1}{4\pi}
        \big(1+L_f\big)
        V(x)
        \left(
            \frac{4}{\aa }
            +
            \frac{16\dl^2}{\aa^3}
            +
            \frac{12\dl^2}{\aa^2}|x-y|
        \right)
        e^{-\frac{\aa}{2} |x-y|}V(y).
    \end{equation*}
    Making use of the fact that $\sfD_{\beta , V}(\dl)(x,y) = 0$ for $(x,y)\notin\cU\cup\cV$
    with $\cU,\cV$ as in~\eqref{eq:UV} and performing the analysis as in the {\it Step 1} of the proof for Proposition~\ref{prop:expansion} we get
    \begin{equation*}
            |\p^2_{\dl\dl}\sfD_{\beta , V}(\dl)(x,y) |
            \leq
            \frac{C}{4\pi}
            \big(1+|x-y|\big)
            e^{-\frac{\aa }{4} |x-y|},
    \end{equation*}
    with some constant $C = C(\aa,f) > 0$.
    By the Schur test we obtain
    \begin{equation*}
    \begin{split}
         \left\|\frac{\wDb(\dl ')- \wDb(\dl )}{\dl '-\dl } - \p_ \dl \wDb (\delta ) \right\|
        & \leq
        |\dl - \dl'|
        \frac{C}{4\pi}
        \int_{\R^2}(1+|x|)e^{-\frac{\aa }{4} |x|}\dd x \\
        & =
        |\dl - \dl'|
        C
        \left(\frac{2}{\aa} + \frac{8}{\aa^2}\right ).
    \end{split}
    \end{equation*}
    Therefore, the convergence~\eqref{eq-2D} is verified.

    Furthermore, the subset of $\dR^2\times\dR^2$ determined  by $\{(x,y) \,:\, x,y\in \dR^2 \,\wedge \, \p_\dl\wDb(\delta)(x,y) \ne 0
    \}$
    can be covered
    by two (intersecting) sets $\cU$ and $\cV$ defined as in~\eqref{eq:UV}.
    %
    Using the inequalities
    $|r_\beta(x) - r_\beta(y)| \ge |x-y|$ and
    $\,\frac14|x-y| \ge \frac{|x| + |y|}{4} - \frac12R$ for $(x,y)\in\cU\cup\cV$
    we get for all $\beta \in (0,1]$ the estimate
    %
    %
    \begin{equation*}
        |\p_\dl\wDb(\delta)(x,y)| \le
        \frac{e^{\aa R} \dl}{2\pi\aa}\left[2+ L_f\right ] e^{-\frac{\aa}{4}(|x|+|y|)}.
    \end{equation*}
    Hence, by the Schur test we find
    \begin{equation*}\label{key}
    \begin{split}
        \|\p_\dl\wDb(\dl)\|&
        \le
        \frac{e^{\aa R} \dl}{2\pi\aa}
        \left[2+ L_f\right ]
        \sup_{x\in\dR^2} \int_{\dR^2} e^{-\frac{\aa}{4}(|x|+|y|)}\dd y\\
        & = 
        \frac{e^{\aa R} \dl}{\aa}
        \left[2+ L_f\right ]
        \int_0^\infty e^{-\frac{\aa}{4} r}r\dd r
         =
        \frac{16e^{\aa R} \dl}{\aa^3}
        \left[2+ L_f\right ],
    \end{split}
    \end{equation*}
    and the claim of (ii) follows.
\end{proof}
In what follows we employ for $V \equiv 1$ the shorthand notation $\sfD^{(1)}(\delta) :=    \sfD^{(1)}_1(\delta)$.

\begin{cor}\label{cor:cD}
    The integral kernel $\sfD^{(1)}(\delta)(\cdot,\cdot)$
    in~\eqref{eq:kernel2} with $\delta\in[0,1]$
    and $V\equiv 1$ satisfies
    \begin{equation}\label{eq:integrability}
    \cD_{\aa,f}(\delta) :=
    2\aa^3
    \int_{\dR^2}\int_{\dR^2}\sfD^{(1)}(\delta)(x,y)\dd x \dd y < \infty.
    \end{equation}
    In addition, the function $[0,1]\ni\delta\mapsto\cD_{\aa,f}(\delta)$
    is continuous.
\end{cor}
\begin{proof}
	Note that there exists an integrable majorant for the integrand in~\eqref{eq:integrability}
    with $\dl \in [0,1]$ given by
    \[
        \dR^2\times\dR^2\ni (x,y) \mapsto
        2\aa^3 C_1 G_{\frac{\aa}{4}}(x-y)
        \left [1+ \left (\frac{\aa}{4} + \frac12\right ) |x-y|\right ]\chi_{\cT}(x,y),
    \]
    where $C_1$ is as in~\eqref{eq:pntws_bnd}, $\cT = (\cS\times\dR^2)\cup(\dR^2\times\cS)$
    and $\chi_\cT\colon \dR^2\times\dR^2\arr \{0,1\}$ is the characteristic function of $\cT$.
    Hence, finiteness of $\cD_{\aa,f}(\delta)$
    directly follows from the asymptotic behaviour of $ G_{\frac{\aa}{4}}(\cdot)$.
    Furthermore, taking into account the pointwise continuity of the integrand
    in~\eqref{eq:integrability}  with respect to $\delta$, continuity of $[0,1]\ni\delta\mapsto\cD_{\aa,f}(\delta)$
    is a consequence of the dominated convergence theorem.
\end{proof}
Finally, we obtain
an alternative formula for $\cD_{\aa,f}(0)$
in terms of the Fourier transform of $f$. In the proof
of this proposition we use an integral representation of the
relativistic Schr\"odinger operator; see~\cite{IT93} and also~\cite[\S 7.12]{LL01}.
\begin{prop}\label{prop:fractional}
    The value $\cD_{\aa,f} := \cD_{\aa,f}(0)$  in~\eqref{eq:integrability}  with $\delta = 0$
    can be represented as
    \[
        \cD_{\aa,f} =
        \int_{\dR^2} |p|^2\left (\aa^2 - \frac{2\aa^3}{\sqrt{4|p|^2 + \aa^2} + \aa}\right) |\wh f(p)|^2\dd p >  0.
    \]
\end{prop}
\begin{proof}
    First, we decompose $\cD_{\aa,f}$
    as $\cD_{\aa,f} = \cD_{\aa,f}^{(1)} - \cD_{\aa,f}^{(2)}$ with
    \begin{equation*}\label{key}
    \begin{split}
        \cD_{\aa,f}^{(1)} & := \frac{\aa^3}{2} \int_{\dR^2}\int_{\dR^2}
        \frac{e^{-\frac{\aa}{2}|x-y|}}{4\pi|x-y|}
        \left (|\nabla f(x)|^2 + |\nabla f(y)|^2  \right )  \dd x \dd y,\\[0.4ex]
        \cD_{\aa,f}^{(2)} & := \frac{\aa^3}{2}\int_{\dR^2}\int_{\dR^2}
        \frac{e^{-\frac{\aa}{2}|x-y|}}{4\pi|x-y|}
        \frac{|f(x) - f(y)|^2}{|x-y|^2} \left (\aa |x-y|+2\right ) \dd x \dd y.
    \end{split}
    \end{equation*}
    Then, we find by elementary computations
    \begin{equation}\label{eq:D1}
    \begin{split}
        \cD_{\aa,f}^{(1)} & =
        \aa^3\left ( \int_{\dR^2} |\nabla f(x)|^2\dd x\right )
        \left (
        \int_{\dR^2}
        \frac{e^{-\frac{\aa}{2}|y|}}{4\pi|y|} \dd y\right )\\
        & =
        \frac{\aa^3}{2}\left ( \int_{\dR^2} |p|^2 |\wh f(p)|^2\dd p\right )
        \left (
        \int_{\dR_+}
        e^{-\frac{\aa}{2} r} \dd r\right )
         =
        \aa^2\int_{\dR^2} |p|^2 |\wh f(p)|^2\dd p.
    \end{split}
    \end{equation}
    Next, using the identities~\cite[Eq. (2.2) and (2.4) for $d =2$]{IT93}
    we get
    \begin{equation*}\label{key}
    \begin{split}
        \int_{\dR^2}
        \left (\sqrt{|p|^2 + \tfrac14\aa^2} - \tfrac12\aa\right )|\wh f(p)|^2 \dd p
        & =
        -\int_{\dR^2}\int_{\dR^2}
        \left (f(x)f(y) - f(x)^2\right )n(x-y)\dd x\dd y\\
        & = \frac12\int_{\dR^2}\int_{\dR^2}
        \left| f(x) - f(y)\right|^2 n(x-y)\dd x\dd y,
    \end{split}
    \end{equation*}
    for $n(\cdot)$ given by
    \begin{equation*}\label{key}
    \begin{split}
        n(x) & =
        2(2\pi)^{-3/2} \left (\frac{\aa}{2}\right )^{3/2}
        |x|^{-3/2}K_{3/2}\left (\frac{\aa}{2}|x|\right )\\
        & =
        2(2\pi)^{-3/2} \left (\frac{\aa}{2}\right )^{3/2}
        |x|^{-3/2} \frac{\left (\frac{\pi}{2}\right )^{1/2}
            \exp\left (-\frac{\aa}{2}|x|\right )
            \left (\frac{2}{\aa|x|}+1\right )}
        {\left (\frac{\aa}{2}|x|\right )^{1/2}}\\
        & =
        \frac{\exp\left (-\frac{\aa}{2}|x|\right )}{4\pi|x|}\frac{1}{|x|^2}\left (2+\aa|x|\right ),
    \end{split}
    \end{equation*}
    where in between we used  the representation
    \begin{equation*}\label{key}
        K_{3/2}(x) =
        \frac{\left (\frac{\pi}{2}\right )^{1/2}
            \exp\left (-x\right )
            \left (\frac{1}{x}+1\right )}{x^{1/2}}
    \end{equation*}
    for the modified Bessel function $K_{3/2}(\cdot)$ of order $\nu = \frac32$.
    Hence, we get
    \begin{equation}\label{eq:D2}
        \cD^{(2)}_{\aa,f}
        =
        \aa^3
        \int_{\dR^2}
        \left (\sqrt{|p|^2 + \tfrac14\aa^2} - \tfrac12\aa\right )|\wh f(p)|^2\dd p
        =
        2\aa^3
        \int_{\dR^2}
        \frac{|p|^2}{\sqrt{4|p|^2 + \aa^2} + \aa} |\wh f(p)|^2\dd p.
    \end{equation}
    Finally, combining~\eqref{eq:D1} and~\eqref{eq:D2} we obtain
    \begin{equation*}\label{key}
        \cD_{\aa,f} = \cD^{(1)}_{\aa,f} - \cD^{(2)}_{\aa,f}
        = \int_{\dR^2}|p|^2\left (\aa^2 -
        \frac{2\aa^3}{\sqrt{4|p|^2 + \aa^2} + \aa}\right )|\wh f(p)|^2
        \dd p.
    \end{equation*}
    In particular, $\cD_{\aa,f} > 0$ follows
        from positivity almost everywhere in $\dR^2$ of the expression in the round
        brackets standing in the integrand in the above
        formula.
\end{proof}

\section{Proof of Theorem~\ref{thm:main}}
\label{sec:proof}

We split the proof of the main result into three steps.

\noindent {\it\underline{Step 1: Spectral equation.}}
In order to derive the spectral equation, we introduce an auxiliary
function\footnote{Introducing $V$ is a purely technical step, needed
for a regularization. The final result does not depend on the
particular choice of $V$.} $V(x) := \e^{\frac{\aa}{4} |x|}$, $\aa >
0$. Note that $V$ satisfies the growth condition specified in
Proposition~\ref{prop:expansion} with $c = 1$ and moreover $V^{-1}
\in L^2(\dR^2)\cap L^\infty(\dR^2)$.  Furthermore, we associate with the kernel
in~\eqref{eq:kernel1} the operator $\wDb(\delta)\in\cB(L^2(\dR^2))$
as in Proposition~\ref{prop:expansion}. Recall that the operator
$\wDb(\delta)$ admits the representation
\begin{equation}\label{eq:expansion_recall}
    \wDb(\delta) =
    \sfD^{(1)}_V(\delta )\beta^2 + \wDb^{(2)}(\delta )\beta^4\,,
\end{equation}
where $\sfD^{(1)}_V(\delta), \wDb^{(2)}(\delta) \in \cB(L^2(\dR^2))$
and we also have $\|\wDb(\delta)\| = \cO_{\rm u}(\beta^2)$ as $\beta
\arr 0+$. Next, we define the product
\begin{equation*}
    \wB (\delta ):= V^{-1} \sfB_\aa (\delta ) V^{-1},\qquad
    \delta > 0,
\end{equation*}
where $\sfB_\aa(\dl)$ is as in~\eqref{eq:B}.
The spectral condition~\eqref{eq:reformulation} can be rewritten as
\begin{equation*}\label{eq-spectral2}
    \forall\, \kp > \frac{\aa}{2},\qquad
    \dim\ker\left (\Op + \kp^2\right )
    =
    \dim\ker\left (\sfI -  \aa\wB(\dl)\wDb(\dl)\right ).
\end{equation*}
To compute the dimension of $\ker(\sfI -\aa\wB(\dl)\wDb(\delta))$ we investigate the asymptotic behaviour of $\wB(\dl)$ as $\dl\arr 0+$.
First, we observe that the decomposition in Lemma~\ref{le-decompoB} yields
\begin{equation*}
    \wB(\delta)
    =
    \frac{\aa^2}{2} V^{-1} \sfR(\delta^2) V^{-1}+\sfN_{\aa,V}(\delta ) \,,
\end{equation*}
where $\sfN_{\aa,V}(\delta ) := V^{-1}\sfN_\aa(\delta )V^{-1}$. Lemma~\ref{le-decompoB}\,(iii) implies that
$\dR_+\ni\delta \mapsto\sfN_{\aa,V}(\delta )$ is real analytic.
Observe that $\sfR(\delta^2)$ is an integral operator with the kernel $\frac{1}{2\pi }K_0 (\delta |x-y|)$,
where $K_0(\cdot)$ is the modified Bessel function
of the second kind and order zero; \cf~\cite[\S 9.6]{AS64}.
The function $K_0$ admits an asymptotic expansion
(see \cite[Eq. 9.6.13]{AS64})
\begin{equation}\label{eq:K0asymp}
    K_0 (z)=
    -\ln \frac{z}{2} - \gamma +\cO(z^2\ln z), \qquad z\arr 0+\,,
\end{equation}
where $\gamma \approx 0.577\dots$ is the Euler-Mascheroni constant.
In accordance to the asymptotics~\eqref{eq:K0asymp}, the
operator-valued function $\delta \mapsto \frac{1}{2} V^{-1}
\sfR(\delta^2) V^{-1}$ can be decomposed as follows
\begin{equation*}\label{eq-decomG2}
    \frac{1}{2} V^{-1} \sfR(\delta^2) V^{-1}
    =
     \sfL(\delta ) +  \sfM (\delta ) \,,
\end{equation*}
where
\begin{equation*}\label{eq:sfL}
    \sfL(\delta ):=
    -\frac{\ln \dl}{4\pi }
    \left (\cdot,V^{-1} \right )_{L^2(\dR^2)} V^{-1}
\end{equation*}
and $\sfM (\delta )\colon L^2(\dR^2)\arr L^2(\dR^2)$, $\delta  > 0$, is a bounded integral operator with the kernel
\begin{equation*}\label{eq:sfM}
    \sfM (\delta) (x,y)
    :=
    \frac{1}{4\pi}V^{-1}(x)\left[ K_0 (\delta |x-y|) + \ln \delta \right] V^{-1}(y)\,.
\end{equation*}
Define also the bounded integral operator $\sfM(0)\colon L^2(\dR^2)\arr L^2(\dR^2)$ with the kernel
\begin{equation*}\label{key}
    \sfM(0)(x,y)
        := -\frac{1}{4\pi}V^{-1}(x)\left[\gamma +\ln \frac{|x-y|}{2} \right] V^{-1}(y)\,.
\end{equation*}
Mimicking the arguments from~\cite[Prop. 3.2]{S76}
we conclude that the operator-valued function $(0,\infty)\ni \delta\mapsto \sfM (\delta)$ is real analytic and that
\begin{equation*}\label{eq:M_cont}
    \left \|\sfM (\delta) - \sfM(0)\right \|\arr 0,\qquad \delta\arr 0+.
\end{equation*}
We define the integral operator $\sfM'(\dl)\colon L^2(\dR^2)\arr L^2(\dR^2)$ via the kernel
\begin{equation*}\label{key}
    \sfM'(\dl)(x,y)
    := \frac{1}{4\pi\dl} V^{-1}(x)\big (
    1 -\dl  K_1(\dl|x-y|) |x-y|\big ) V^{-1}(y),
\end{equation*}
where $K_1(\cdot)$ is the modified Bessel function
of the second kind and order $\nu = 1$; \cf~\cite[\S 9.6]{AS64}. Analogously, for the $\sfM(\dl)$ one checks the following convergence
\begin{equation*}\label{eq-M}
    \lim _{\dl'\to \dl }\left\|\frac{\sfM (\dl ')- \sfM (\dl )}{\dl '-\dl } -  \sfM'(\dl ) \right\| = 0.
\end{equation*}
Consequently, $\sfM'  (\dl )$ can be identified with $\partial_{\delta} \sfM (\dl )$.
Furthermore, using the inequality $1 - x K_1(x) < x$
we get by the Schur test
\begin{equation}\label{eq:Mprime}
\begin{split}
    \|\partial_{\delta} \sfM (\dl )\|
    &
    \le
    \frac{1}{4\pi\dl}
    \sup_{x\in\dR^2}
    \int_{\dR^2} e^{-\frac{\aa}{4}|x|}e^{-\frac{\aa}{4}|y|}
    \big| 1 -\dl|x-y| K_1(\dl|x-y|)\big |\dd y \\
    & \le
    \frac{1}{4\pi}
    \sup_{x\in\dR^2}\left (
    e^{-\frac{\aa}{4}|x|}
    \int_{\dR^2} e^{-\frac{\aa}{4}|y|}(|x| +|y|) \dd y\right )\\
    &
    =
    \frac{1}{2}
    \left [
    \left (\sup_{x\in\dR^2} |x|e^{-\frac{\aa}{4}|x|}\right )
    \int_0^\infty  e^{-\frac{\aa}{4} r} r \dd r +
    \int_0^\infty e^{-\frac{\aa}{4} r} r^2\dd r\right ]
     =
    \frac{32}{\aa^3}\big(e^{-1} + 2\big).
\end{split}
\end{equation}
Next, denote
\begin{equation*}\label{key}
    \Gb (\delta )
    :=
    \left (\aa^2\sfM (\delta) + \sfN_{\aa,V}(\delta)\right )\wDb (\delta).
\end{equation*}
real analyticity of  $\wDb (\delta)$, $\sfN_{\aa,V}(\delta)$, and
$\sfM(\dl)$ with respect to $\dl,\beta  \in (0,1)$ implies that
$\Gb (\delta )$ is also real analytic in $\delta,\beta \in
(0,1)$. It follows from the
expansion~\eqref{eq:expansion_recall} and the above estimates that
$\Gb (\delta )$ is a bounded operator, whose norm behaves as $\|
\Gb(\delta)\| = \cO_{\rm u}(\beta^2)$ as $\beta \arr 0+$. Using
Lemma~\ref{le-decompoB}\,(iv), Propositions~\ref{prop:expansion}
and~\ref{prop:analytic}, and the estimate~\eqref{eq:Mprime}~ we get
applying the triangle inequality for the operator norm
\begin{equation}\label{eq:Gprime}
\begin{split}
        \|\p_\dl\Gb (\dl )\|
        &\le
        \big[\aa^2 \|\partial_{\delta} \sfM (\dl )\| + \|\p_\dl\sfN_{\aa,V}(\dl)\|\big ]
        \|\wDb (\dl)\|\\
        & \qquad\quad + \big[\aa^2 \|\sfM (\dl)\|
        + \|\sfN_{\aa,V}(\dl)\|\big]
        \|\p_\dl\wDb (\dl)\|
        = \cO_{\rm u}(1), \qquad \beta \arr 0+.
\end{split}
\end{equation}
Next, for all sufficiently small $\beta > 0$, the operator $\sfI -
\aa\Gb (\delta )$ is invertible  and $\sfI -
\aa\wB(\delta)\wDb(\delta) $ can be factorized as
\begin{equation*}\label{eq-spectral4}
    \sfI -  \aa\wB(\delta)\wDb(\delta)
    =
    (\sfI- \aa\Gb (\delta ))
    \left(\sfI- \Pb (\delta ) \right),
\end{equation*}
where $\Pb (\delta )$ is the rank-one operator given by
\begin{equation*}\label{key}
\begin{split}
    \Pb (\delta )
    & := (\sfI - \aa\Gb (\delta ))^{-1}
                \sfL(\delta ) \aa^3\wDb (\delta ) \\
    & = - \aa^3\frac{\ln \delta }{4\pi }
     \left ( \cdot,\wDb (\delta ) V^{-1} \right )_{L^2(\dR^2)} (\sfI - \aa\Gb (\delta ))^{-1} V^{-1}.
\end{split}
\end{equation*}
Thus, we get for all sufficiently small $\beta > 0$
\begin{equation*}\label{key}
    \forall \delta > 0, \qquad
    \dim\ker\left (\sfI -\aa\wB(\delta)\wDb(\delta)\right )
    =
    \dim\ker\left (\sfI - \Pb(\delta)\right ).
\end{equation*}
Observe that $\dim\ker\left (\sfI - \Pb(\delta)\right ) \in \{0,1\}$. Using the relation $\dim\ker(\sfI - \sfP) = 1$
\iff~$\Tr\sfP =  1$ (true for any rank-one
operator $\sfP$), we find that  $\dim\ker (\sfI - \Pb (\delta )) = 1$
\iff
\begin{equation}\label{eq-spectralfinal}
 \boxed{4\pi +\aa^3\ln \delta
    \left ( \wDb (\delta )V^{-1},
    (\sfI- \aa\Gb (\delta ))^{-1}V^{-1}\right )_{L^2(\dR^2)} =0 \,.}
\end{equation}
In view of this reduction,
for all sufficiently small $\beta > 0$, each solution $\delta > 0$ of the equation~\eqref{eq-spectralfinal} corresponds to a simple eigenvalue $-\frac{\aa^2}{4} - \delta^2$
of $\Op$.

\vspace{0.6ex}

\noindent\underline{\emph{Step 2: Existence and uniqueness of solution for~\eqref{eq-spectralfinal}.}}\,
Define the function
\begin{equation*}\label{key}
    \eta_\aa(\beta,\delta)
    :=
    2\aa^3
    \left ( \wDb (\delta ) V^{-1} ,
    (\sfI- \aa\Gb (\delta ))^{-1} V^{-1} \right )_{L^2(\dR^2)}.
\end{equation*}
We remark that the function $\eta_\aa(\cdot,\cdot)$ is real analytic
in $\dl, \beta > 0$ lying in a sufficiently small right neighbourhood of
the origin, thanks to real analyticity with respect to the same
parameters of the operator-valued functions $\wDb(\delta)$ and
$\Gb(\delta)$; see Proposition~\ref{prop:analytic} and the
discussion in {\it Step 1}. The spectral
condition~\eqref{eq-spectralfinal} can be equivalently written as
\begin{equation}\label{eq:spectral_easy}
        \eta_\aa(\beta,\delta) = -\frac{8\pi}{\ln \delta}.
\end{equation}
Applying the Neumann series argument and
using that $\|\Gb(\delta)\| = \cO_{\rm u}(\beta^2)$ ($\beta \arr 0+$)  we find
\begin{equation*}\label{key}
    \|\left (\sfI - \aa\Gb (\delta )\right )^{-1} - \sfI\| = o_{\rm u}(1),
    \qquad \beta\arr 0+.
\end{equation*}
Hence, we conclude from $\|\wDb(\delta)\| = \cO_{\rm u}(\beta^2)$ as
$\beta\arr 0+$ that $\eta_\aa(\beta,\delta) = \cO_{\rm u}(\beta^2)$ as $\beta\arr 0+$. Combining the expansion~\eqref{eq:expansion_recall} and Corollary~\ref{cor:cD} we arrive at
\begin{equation}\label{eq:expansion_eta}
\begin{split}
    \eta_\aa(\beta,\delta)
    & =
    2\aa^3
    \beta^2\int_{\dR^2}\int_{\dR^2}
    \sfD^{(1)}_V(\delta)(x,y) V^{-1}(x)V^{-1}(y) \dd x \dd y + \cO_{\rm u}(\beta^4) \\
    & =
    2\aa^3
    \beta^2\int_{\dR^2}\int_{\dR^2}
    \sfD^{(1)}(\dl)(x,y) \dd x \dd y + \cO_{\rm u}(\beta^4)\\
    & =
    \cD_{\aa,f}(\delta)\beta^2 + \cO_{\rm u}(\beta^4),\qquad \beta\arr 0+.
\end{split}
\end{equation}
Since $\cD_{\aa,f} = \cD_{\aa,f}(0)  > 0$ by Proposition~\ref{prop:fractional},
Corollary~\ref{cor:cD} yields
that $\eta_\aa(\beta,\delta) > 0$  for all sufficiently small $\delta, \beta > 0$. The continuous function $(0,1)\ni \delta\mapsto -\frac{8\pi}{\ln \delta}$
vanishes as $\delta\arr 0+$ and its range coincides with $(0,\infty)$. Hence, for all sufficiently small $\beta > 0$ the equation~\eqref{eq:spectral_easy} has at least one solution $\delta(\beta) > 0$, which satisfies $\delta(\beta)\arr 0+$
taking the lower bound in Proposition~\ref{prop:lower_bound} into account. In particular, we proved that $\#\sd(\Op) \ge 1$
holds for all sufficiently small $\beta > 0$.

It remains to show that in fact for all sufficiently small $\beta >0$ holds $\#\sd(\Op) = 1$. Indeed, the equation~\eqref{eq:spectral_easy} can be rewritten as
\begin{equation*}\label{eq:spectral_easy2}
    \wt\eta_{\aa}(\beta,\delta)  = 0 \qquad
    \text{with}~~~
    \wt\eta_{\aa}(\beta,\delta) := \exp\left (-\frac{8\pi}{\eta_\aa(\beta,\delta)}\right) - \delta.
\end{equation*}
Suppose that $\beta > 0$ is small enough and that $\wt\eta_{\aa}(\beta,\delta) = 0$ has two solutions
$\delta_1,\delta_2 \in (0,1)$ such that $\delta_1 < \delta_2$.
By Rolle's theorem there exists a point $\delta_\star  \in (\delta_1,\delta_2)$ such that
\begin{equation}\label{eq:Rolle}
    (\p_\delta \wt\eta_{\aa})(\beta,\delta_\star) = 0.
\end{equation}
Computing the partial
derivative of $\wt\eta_\aa$ with respect to $\delta$ we get
\begin{equation}\label{eq:deriv_wt_eta}
    \p_\delta\wt\eta_{\aa}(\beta,\delta)
    =
    \frac{8\pi\p_{\delta}\eta_\aa(\beta,\delta)}{\eta_\aa(\beta,\delta)^2}\exp\left (-\frac{8\pi}{\eta_\aa(\beta,\delta)}\right) - 1.
\end{equation}
Differentiating the operator-valued function
$(\sfI - \sfG_\aa(\dl))^{-1}$ with respect to $\dl$ we find
\begin{equation*}\label{der:G}
\begin{split}
    & \lim_{\dl'\to \dl } \frac{( \sfI-\aa \Gb (\dl '))^{-1}- (\sfI-\aa \Gb (\dl ))^{-1}}{\dl'-\dl }\\
    &\qquad = \aa\lim_{\dl'\to \dl }
    (\sfI-\aa \Gb (\dl'))^{-1}
    \frac{\Gb (\dl ')-\Gb (\dl) }{\dl'-\dl }(\sfI -\aa \Gb (\dl))^{-1} \\
    &\qquad\qquad =
    \aa(\sfI-\aa \Gb (\dl ))^{-1} \p_ \dl\Gb (\dl )
    (\sfI-\aa \Gb (\dl ))^{-1}.
\end{split}
\end{equation*}
Hence, differentiating the scalar function $\eta_\aa$ with respect to $\delta$ and applying Propositions~\ref{prop:expansion},~\ref{prop:analytic} and the estimate~\eqref{eq:Gprime}, we end up with
\[
\begin{split}
    \p_{\delta}\eta_\aa(\beta,\delta)
    & =
    2\aa^3
    \left (\p_\delta \wDb (\delta ) V^{-1} ,
    (\sfI- \aa\Gb (\delta ))^{-1} V^{-1} \right )_{L^2(\dR^2)}\\
    & \qquad +
    2\aa^4
    \left ( \wDb (\delta ) V^{-1} ,
    (\sfI-\aa \Gb (\dl ))^{-1}
    (\p_\delta\Gb (\delta )) (\sfI- \aa\Gb (\delta ))^{-1} V^{-1} \right )_{L^2(\dR^2)}\\
    & =  \cO_{\rm u}(1),
    \quad\beta\arr 0+.
\end{split}
\]
Eventually, we derive from~\eqref{eq:deriv_wt_eta} that $\p_\delta\wt\eta_{\aa}(\beta,\delta)  = -1 + o_{\rm u}(1)$ as $\beta\arr 0+$, which contradicts to~\eqref{eq:Rolle} for all sufficiently small
    $\beta > 0$.

\vspace{0.6ex}

\noindent\underline{\emph{Step 3: Asymptotic expansion.}}\,
Let $\dl(\beta) > 0$ be the unique solution of~\eqref{eq:spectral_easy} for sufficiently small
$\beta > 0$.
Substituting the expansion~\eqref{eq:expansion_eta} into
the spectral condition~\eqref{eq:spectral_easy}
and making an additional use of $\delta(\beta) = o(1)$
(as $\beta\arr 0+$) we get
\begin{equation*}\label{key}
    8\pi + \ln \delta(\beta)\beta^2
    \cD_{\aa,f}(\delta(\beta))
    +   o(\ln \delta(\beta)\beta^2) = 0,\qquad
    \beta \arr 0 +.
\end{equation*}
Applying Corollary~\ref{cor:cD} we obtain
\[
    8\pi + \ln \delta(\beta)\beta^2
    \cD_{\aa,f} + o(\ln \delta(\beta)\beta^2) = 0,\qquad\beta \arr 0 +.
\]
Hence, we deduce
\begin{equation}\label{eq-asymdelta}
    \delta(\beta) =
    \exp \left( - \frac{8\pi}{\cD_{\aa,f}\beta^2}\right) \left (1+o(1) \right ),\qquad\beta \arr 0 +.
\end{equation}
Finally, the asymptotic expansion of $\lm_1^\aa(\beta)$ in~\eqref{eq:main_expansion}
follows from~\eqref{eq-asymdelta} and the identity
$\lm_1^\aa(\beta) = -\frac14\aa^2 - \delta^2(\beta)$.\qed

\section{Discussion}\label{sec:dis}

The main result of this paper
might be possible to extend for less regular $f$ with a non-compact support. A natural limitation of admissible generalizations is finiteness of the constant $\cD_{\aa,f}$
in Theorem~\ref{thm:main}.

Apparently, a similar asymptotic analysis can be performed in space
dimensions $d \ge 4$, where not much is known apart from the result
in~\cite{LO16} mentioned above. We note that a convincing physical
motivation is missing in this case, so far at least, and also one
can expect here that for all sufficiently small $\beta>0$ the
discrete spectrum would be empty.

It is also worth noting that analogous spectral problem can be
considered for the Robin Laplacian in a locally perturbed
half-space. In view of~\cite{EM14} one may expect that the existence
of the unique bound state for all sufficiently small $\beta > 0$
will depend on the function $f$, defining the profile of the
deformation. However, the technique to deal with the asymptotic
analysis should be different for the Robin spectral problem, because
a Birman-Schwinger-type principle with an explicitly given integral
operator is not available in this setting.

Finally, let us point out that in the present paper we have not
touched the case where the interaction support is a topologically
non-trivial surface which could be regarded as a certain analogue of
spectral analysis in infinite, topologically nontrivial
layers~\cite{CEK04}. It is not so clear to what extent the main
result and the technique of the present paper can be generalized to
include such more involved geometries.

\subsection*{Acknowledgements}
P.E. and V.L. acknowledge the support by the grant No.~17-01706S of
the Czech Science Foundation (GA\v{C}R). S.K and V.L. acknowledge
the support by the grant DEC-2013/11/B/ST1/03067 of the Polish
National Science Centre. Moreover, V.L. is grateful to the University of Zielona G\'{o}ra for the hospitality during a visit in February 2016, where a part of this paper was written.
S.K. thanks the Department of Theoretical Physics, NPI CAS in \v{R}e\v{z}, for the hospitality
during a visit in November 2016.

%
\newcommand{\etalchar}[1]{$^{#1}$}

\end{document}